\documentclass[12pt]{article}

\usepackage{amsmath,amsthm}
\usepackage{amssymb}
\usepackage{thmtools}

\usepackage[T1]{fontenc}

\usepackage[colorlinks=true, urlcolor=blue, linkcolor=blue,
citecolor=black]{hyperref}

\theoremstyle{plain}
\declaretheorem[title=Theorem]{theorem}
\declaretheorem[title=Lemma,sibling=theorem]{lemma}

\declaretheorem[title=Proposition,sibling=theorem]{proposition}

\theoremstyle{definition}

\declaretheorem[numbered=no,title=Remark]{remark*}

\newcommand{\R}{\mathbb{R}}

\newcommand{\T}{\mathbb{T}}

\newcommand{\C}{\mathcal{C}}
\newcommand{\eps}{\varepsilon}

\usepackage{xspace}
\newcommand{\POne}{\text{(P1)}}
\newcommand{\PTwo}{\text{(P2)}}
\newcommand{\QOne}{(Q1)}
\newcommand{\QTwo}{(Q2)}

\usepackage[english]{babel}

\addto\extrasenglish{%

}

\usepackage[backgroundcolor=white, bordercolor=blue,
linecolor=blue]{todonotes}

\AtBeginDocument{%
 \setlength\abovedisplayskip{2pt}
 \setlength\belowdisplayskip{0pt}
}

\begin{document}

\title{Homogenization of L\'{e}vy-type operators with oscillating coefficients}

\author{M. Kassmann\footnote{Fakult\"{a}t f\"{u}r Mathematik, Universit\"{a}t
Bielefeld, Postfach
100131, D-335001 Bielefeld, Germany, email:
\emph{moritz.kassmann@uni-bielefeld.de}},\ \
A. Piatnitski\footnote{The Arctic University of Norway, Campus in
Narvik, P.O. Box 385, Narvik 8505, Norway \ and \ Institute for Information Transmission Problems
of RAS,127051 Moscow, Russia, email: \emph{apiatnitski@gmail.com}},\ \
E. Zhizhina\footnote{Institute for Information Transmission
Problems of RAS, 127051 Moscow, Russia, email: \emph{ejj@iitp.ru}}}

\date{\today}
\maketitle
\parskip 0.1 truein
\parindent0ex
%
%


\begin{abstract}
The paper deals with homogenization of L\'{e}vy-type operators with  rapidly
oscillating
coefficients. We consider cases of periodic and random statistically
homogeneous micro-structures and show that in the limit we obtain a
L\'{e}vy-operator. In the periodic case we study both symmetric and
non-symmetric kernels whereas in the random case we only investigate symmetric
kernels. We also address a nonlinear version of this homogenization problem.
\end{abstract}



\emph{Acknowledgement:} Financial support from the German Science Foundation
via Sonderforschungsbereich 1283 is gratefully acknowledged.

\section{Introduction}

The paper deals with a homogenization problem for  L\'{e}vy-type operators of
the form
\begin{align}\label{eq:def-Leps}
L^\eps u (x)=\int\limits_{\mathbb R^d}
\frac{u(y)-u(x)}{|x-y|^{d+\alpha}} \, \Lambda^\eps \left(x,y\right) dy \qquad
(x \in \R^d)\,,
\end{align}
where $\alpha \in (0,2)$ is fixed, $u\in L^2(\mathbb R^d)$ and $\eps>0$ is
a small parameter. We will
study various assumptions on the function $(x,y) \mapsto \Lambda^\eps
\left(x,y\right)$. Throughout the article we assume
\begin{align}\label{eq:ellipticity}
\gamma^{-1}\leq \Lambda^\eps(x,y)\leq \gamma \qquad (x, y \in \R^d)
\end{align}
for some $\gamma > 1$, which can be seen as an ellipticity assumption.
Particular cases that we cover include $\Lambda^\eps(x,y)= \Lambda\left(\frac
x\eps, \frac y\eps\right)$ resp. $\Lambda^\eps(x,y)= \Lambda\left(\frac
x\eps, y\right) + \Lambda\left(x, \frac
y\eps \right)$, where
$(\xi,\eta) \mapsto \Lambda(\xi,\eta)$ is symmetric and periodic both in $\xi$
and $\eta$. Note that we also deal with some classes of non-symmetric kernels
and of random symmetric kernels. Moreover, the approach allows to treat
nonlinear nonlocal operators such as the fractional $p$-Laplace operator.

\medskip

Given $\eps > 0$, we first introduce a positive
self-adjoint extension of the operator $-L^\eps$ and then study the following
homogenization problem:

\hspace*{3ex} \begin{minipage}{0.85\textwidth}
Find an operator $L^0$
 such that for any $m>0$ and for any $f\in L^2(\mathbb R^d)$ the solutions
$u^\eps$ of the equations
 $-L^\eps u^\eps+ m u^\eps=f$ converge, as $\eps\to0$, to the solution of the
equation $-L^0 u+mu=f$.
\end{minipage}

Given $\eps > 0$, the operator $L^\eps$ describes a jump process in a
non-homogeneous medium with a periodic micro-structure. For $\Lambda^\eps=1$
this operator coincides, up to a multiplicative constant, with the fractional
Laplacian
$(-\Delta)^{\alpha/2}$, which is the infinitesimal generator of the
rotationally symmetric $\alpha$-stable process \cite{Sat13}. As we show in this
work, the computation of the homogenization limit for a nonlocal operator of
fractional order $\alpha$ of differentiability is rather different from the
corresponding object for differential operators. In the symmetric case, it
turns out that the effective jump rate is given as a simple average whereas this
is easily seen to be false for differential operators. In the non-symmetric
case treated in \autoref{theo:nonsym}, however, we face similar phenomenons as
in the case of local differential operators.

\medskip

Let us formulate our main results. We consider three different
settings. Note that throughout the paper we deal with bilinear forms
resp. weak solutions because, the expression $L^\eps u (x) $ might not exist
point-wisely, even for $u \in C^\infty_0(\R^d)$ and $\Lambda^\eps$ as in the
aforementioned example. Some additional regularity of $\Lambda^\eps$ at the
diagonal $x=y$ would be needed otherwise. Let us now present the three settings
of our study.

\emph{(I) Symmetrizable and symmetric periodic kernels:} Here we assume
that $\Lambda^\eps$ is a  positive function satisfying one of the
following two
conditions.

\begin{itemize}
 \item[{\POne}] Product structure: We assume
\begin{align}\label{eq:case_product}
\Lambda^\eps(x,y)= \lambda\left(\frac x\eps\right)\mu\left(\frac y\eps\right)
\end{align}
with $\lambda$ and $\mu$ being $1$-periodic in each direction and satisfying
 \begin{align}\label{eq:elli_prod}
\gamma^{-1}\leq \lambda(\xi)\leq \gamma,\quad \gamma^{-1}\leq \mu(\eta)\leq
\gamma.
\end{align}
\item[{\PTwo}] Symmetric locally periodic kernels: We assume
\begin{align}\label{eq:case_locally}
\Lambda^\eps(x,y)= \Lambda\left(x,y,\frac x\eps, \frac y\eps\right)
\end{align}
with a function $\Lambda(x,y,\xi,\eta)$ that is continuous in $(x,y)$,
periodic measurable in $(\xi,\eta)$,
and satisfies the following conditions:
\begin{align*}
\left.
\begin{array}{rl}
\Lambda(x,y,\xi,\eta) &=\Lambda(y,x,\eta,\xi) \\
\gamma^{-1} &\leq \Lambda(x,y,\xi,\eta)\leq \gamma
\end{array}
\right\}
\text{ for all }
x,\,y,\,\xi,\,\eta\in \mathbb R^d \,.
\end{align*}
\end{itemize}

In order to characterize the limit behaviour of $u^\eps$ we introduce an
operator
\begin{align}\label{eq:limit_operator}
L^0 u(x)=\int\limits_{\mathbb R^d} \frac{\displaystyle \Lambda^{\rm eff}
\big(u(y)-u(x)\big)}{|y-x|^{d+\alpha}}\:dy
\end{align}
where
\begin{align}\label{eq:Lambda-eff-sym}
\Lambda^{\rm eff}(x,y) =
\begin{cases}
\bigg(\int_{[0,1]^d}\frac{\mu(\xi)}{\lambda(\xi)}d\xi\bigg)^{-1}
\bigg(\int_{[0,1]^d}\mu(\xi)d\xi\bigg)^2 &\text{ in Case } {\POne} \,, \\
\int_{[0,1]^d}\int_{[0,1]^d}\Lambda(x,y,\xi,\eta)\,d\xi
d\eta &\text{ in Case } {\PTwo} \,.
\end{cases}
\end{align}

\begin{theorem}\label{theo:symm-kernels}
Assume that one of the conditions {\POne}, {\PTwo} holds true. Let $m > 0$. Then
for every $f\in L^2(\mathbb R^d)$ the solution $u^\eps$
of equation
\begin{align}\label{eq:ori_eqn}
(L^\eps-m)u^\eps=f
\end{align}
converges strongly in $L^2(\mathbb R^d)$ and weakly in $H^{\alpha/2}(\mathbb R^
d)$ to the solution $u^0$ of the equation
\begin{align}\label{eq:lim_eqn}
(L^0-m)u^0=f\,.
\end{align}
\end{theorem}

\begin{remark*}
\begin{enumerate}
 \item[(i)] Case {\PTwo} contains the particular case of pure periodic
coefficients, which we have mentioned above. If one assumes
$\Lambda^\eps(x,y)=
\Lambda\left(\frac x\eps, \frac y\eps\right)$ with a function
$\Lambda(\xi,\eta)$ that is periodic both in $\xi$ and $\eta$ and satisfies for
all $\xi,\,\eta\in \mathbb R^d$ the conditions
$\Lambda(\xi,\eta)=\Lambda(\eta,\xi)$ $\gamma^{-1}\leq \Lambda(\xi,\eta)\leq
\gamma$, then this case covered by {\PTwo}.
\item[(ii)] In Case {\POne} the function $\Lambda^{\rm eff}$ is
constant, i.e., the operator $L^0$ is invariant under translations.
\item[(iii)] In Case {\PTwo} we can choose $\Lambda^\eps(x,y)=
a(x,y) \Lambda\left(\frac x\eps, \frac y\eps\right)$ with a function
$a:\R^d \times \R^d \to [a_0, a_1] \subset (0,\infty)$. In this case
\begin{align*}
\Lambda^{\rm eff}(x,y) =
a(x,y) \int_{[0,1]^d}\int_{[0,1]^d}\Lambda(\xi,\eta)\,d\xi
d\eta\,,
\end{align*}
i.e., the limit operator $L^0$ is a nonlocal operator with bounded and
measurable coefficients.
\end{enumerate}
\end{remark*}

\autoref{theo:symm-kernels} deals with linear nonlocal operators. The methods
of its proof can be applied to nonlinear problems, too. Let us provide a
nonlocal analog of \autoref{theo:symm-kernels}. Assume $p >1$. Given $\eps
> 0$, define a nonlinear version $L_p^\eps$ of $L^\eps$ by
\begin{align}\label{eq:nonlin-Leps}
L_p^\eps u(x)=\int\limits_{\mathbb R^d}
\frac{|u(y)\!-\!u(x)|^{p-2} (u(y)-u(x))}{|x-y|^{d+
\alpha}}\Lambda^\eps(x,y)  \,dy \qquad (x \in \R^d) \,.
\end{align}

\begin{theorem}\label{theo:nonlin}
Assume that one of the conditions {\POne}, {\PTwo} holds true. Let $m > 0$, $p
> 1$ and $p' = \frac{p-1}{p}$. For any $f\in L^{p'}(\mathbb R^d)$ the solution
$u^\eps$ of equation
\begin{align}\label{eq:nonlin-eps-eq}
  L_p^\eps u - m |u|^{p-2} u =f
\end{align}
converges strongly in $L^p(\mathbb R^d)$ and weakly in
$W^{\frac{\alpha}{p},p}(\mathbb R^d)$, as $\eps\to0$, to the solution $u^0$ of
the equation $L_p^0 u^0 - m |u^0|^{p-2} u^0 =f $, where
\begin{align*}
 L_p^0 u(x)=\int_{\mathbb R^d}
\frac{|u(y)\!-\!u(x)|^{p-2} (u(y)-u(x))}{|x-y|^{d+ \alpha}}
\Lambda^{\rm eff}(x,y) \,dy
\end{align*}
and $\Lambda^{\rm eff}(x,y)$ is as in \eqref{eq:Lambda-eff-sym}.
\end{theorem}

Obviously, \autoref{theo:nonlin} contains \autoref{theo:symm-kernels} because
we could choose $p=2$. Since the proof of \autoref{theo:nonlin} does not
require any new idea, we provide the proof of \autoref{theo:symm-kernels} in
full detail. In  \autoref{subsec:proof-nonlin} we explain how to derive
\autoref{theo:nonlin}.

\medskip \pagebreak[2]

\emph{(II) Symmetric random kernels:}

Let $(\Omega,\mathcal{F},\mathbf{P})$ be a standard probability space and
$(T_y)_{y\in \mathbb R^d}$, be a $d$-dimensional ergodic dynamical system in
$\Omega$; see \autoref{sec:sym_rand} for a detailed definition. As in the case
of deterministic symmetrizable kernels we consider two different setups.

\begin{itemize}
 \item[{\QOne}] Product structure: We assume \eqref{eq:case_product}, where
$\lambda(\xi)$ and $\mu(\xi)$ are realizations
of statistically homogeneous ergodic fields in $\mathbb R^d$. Let
$\omega \mapsto \widehat\lambda(\omega)$ and $\omega \mapsto
\widehat\mu(\omega)$ be random variables such that for some $\gamma>0$ and for
almost every $\omega\in\Omega$
\begin{align}\label{elli_rand}
\gamma^{-1}\leq \widehat\lambda(\omega)\leq \gamma,\qquad \gamma^{-1}\leq
\widehat\mu(\omega)\leq \gamma \,.
\end{align}

Set
\begin{align*}
\lambda(\xi)=\lambda(\xi,\omega)=\widehat\lambda(T_\xi\omega),\qquad
\mu(\xi)=\mu(\xi,\omega)=\widehat\mu(T_\xi\omega) \,.
\end{align*}

The limit  operator takes the form \eqref{eq:limit_operator} with
\begin{align*}
\Lambda^{\rm
eff}=\left(\mathbf{E}\Big\{\frac{\widehat\mu(\cdot)}{\widehat\lambda(\cdot)}
\Big\}
\right)^ {
-1}
\big\{\mathbf{E}\widehat\mu(\cdot)\big\}^2.
\end{align*}

\item[{\QTwo}]  Symmetric random structure: Here, we additionally
assume some topological structure. We assume that $\Omega$ is a metric compact
space. Assume $\mathcal{F}$ is the Borel $\sigma$-algebra of $\Omega$. We
further assume that the group $T_x$ is continuous,  that
$\Lambda=\Lambda(x,y,\omega_1,\omega_2)$ is a continuous function on $\mathbb
R^d\times\mathbb R^d\times\Omega\times\Omega$ and that the following symmetry
conditions is fulfilled:
$\Lambda(x,y,\omega_1,\omega_2)=\Lambda(y,x,\omega_2,\omega_1)$.
 In this case we set
\begin{align}
\Lambda^{\rm eff}(x,y)=\int_\Omega\int_\Omega
\Lambda(x,y,\omega_1,\omega_2)d\mathbf{P}(\omega_1)d\mathbf{P}
(\omega_2)\,.
\end{align}
\end{itemize}

\medskip

\begin{theorem}\label{theo:rand-kernels}
Assume that one of the conditions {\QOne}, {\QTwo} holds true. Let $m >
0$. Almost surely for any $f\in
L^2(\mathbb R^d)$ the solution $u^\eps$
of equation \eqref{eq:ori_eqn} converges strongly in $L^2(\mathbb R^d)$ and
weakly in $H^{\alpha/2}(\mathbb R^
d)$ to the solution $u^0$ of the equation \eqref{eq:lim_eqn}.
\end{theorem}

\medskip

\emph{(III) Non-symmetric kernels:}

One important feature of our approach is that we can allow for certain
non-symmetric kernels in \eqref{eq:def-Leps}. In this case we assume
$0<\alpha<1$. We assume that $\Lambda^\eps$ is a positive function satisfying
$\Lambda^\eps(x,y)= \Lambda\left(\frac x\eps, \frac y\eps\right)$ for a
function
$\Lambda(\xi,\eta)$ that is periodic both in $\xi$ and $\eta$ and satisfies the
following conditions:

\begin{itemize}
\item[(i)] There is $\gamma > 1$ such that $\gamma^{-1}\leq
\Lambda(\zeta,\eta)\leq \gamma$ for all $\zeta$ and
$\eta$.
\item[(ii)] $\Lambda(\zeta,\eta)$ is Lipschitz continuous in each component.
\end{itemize}

As we explain in \autoref{sec:nonsym}, under these conditions  the map $v
\mapsto Lv$ with
\begin{align*}
L(\zeta)=\int\limits_{\mathbb
R^d}\frac{\Lambda(\zeta,\eta)\big(v(\eta)-v(\zeta)\big)}
{|\zeta-\eta|^{d+\alpha}}\,d\eta\,,
\end{align*}
defines an unbounded linear operator $L$ in $L^2(\T^d)$, whose adjoint is given
by
\begin{align*}
L^*q(\zeta)=\int\limits_{\mathbb
R^d}\frac{\big(\Lambda(\eta,\zeta)q(\eta)-\Lambda(\zeta,\eta)q(\zeta)\big)}
{|\zeta-\eta|^{d+\alpha}}\,d\eta \qquad (q \in L^2(\T^d))\,.
\end{align*}

\begin{theorem}\label{theo:nonsym}
For any $f\in L^2(\mathbb R^d)$ the solution $u^\eps$
of equation \eqref{eq:ori_eqn} converges stronlgy in $L^2(\mathbb R^d)$ and
weakly in
$H^{\alpha/2}(\mathbb R^ d)$ to the solution $u^0$
of \eqref{eq:lim_eqn}. Here, the effective jump kernel is given by
$\Lambda^{\rm eff}=\langle
p_0\rangle^{-1}\langle\Lambda p_0\rangle$, where $p_0$ is the
principal eigenfunction of the operator $L^*$ on $\mathbb T^d$, and $\langle\Lambda p_0\rangle
=\int\limits_{\mathbb T^d}\int\limits_{\mathbb
T^d}\Lambda(\xi,\eta)p_0(\xi)\,d\xi d\eta$.
\end{theorem}

\medskip

Let us discuss related articles that deal with homogenization problems for
L\'{e}vy-type operators resp. jump processes. We do not mention the early
fundamental works on homogenization of diffusion-type (differential) operators.
The interested reader is referred to the monographs \cite{JKO94, MR1765047,
MR1798387, braides2005gamma, MR2337848, MR2582099}.

A probabilistic approach to the homogenization problem for
nonlocal operators in non-divergence form is developed in
\cite{MR0488335}, \cite{MR1154844} and in \cite{MR1784743}. An approach based
on PDE methods and viscosity solutions can be found in \cite{MR2560294,
MR2981018}. The PDE method has also been extended to several classes of
nonlinear problems, see \cite{MR3194684, MR2733264, MR3009077}. All these
approaches, like ours,  deal with approximations of the same differentiability
order as the limit operator resp. limit equation. Since one can approximate
diffusions through much simpler objects such as random walks or Markov chains,
it is not surprising that there are also homogenization models for jump
processes that generate a diffusion in the limit, see \cite{San16} or
\cite{PZ2017},

An annealed convergence result for jump processes in random media is contained
in \cite[Theorem 5.3]{MR2565557}. As in our quenched result, no corrector
appears. Convergence in law of jump processes with periodic jump
intensities is also studied in \cite{MR2362589}.   \cite{MR2232732, MR2331266} focus
on homogenization of processes with variable order. Aperiodic fractional obstacle problems are
studied in \cite{MR2729014}.

 The recent papers \cite{MR3668594, BGG18} address problems which,  to a certain extend,
are related to the problems that we consider in the present
work. In these papers the authors focuses on the problem of $H$-compactness of a family of uniformly
elliptic non-local operators and describe a possible structure of any limit point of this family.
Our goal is to show that
for the operators with (locally) periodic and statistically homogeneous
coefficients the whole family of the rescaled operators $G$-converges
and to compute the coefficients of the effective  nonlocal operators. The results of
\cite{MR3668594,
BGG18} imply that in our case there is a non-trivial set of the limit operators with known ellipticity
bounds but leave open the question of their precise shape. Furthermore, we also provide a
quenched convergence result for
random kernels and we treat some non-symmetric cases. Last, apart from the
Gamma-convergence techniques, our proofs are rather different.

\medskip

The organization of the article is simple. \autoref{sec:proof-sym-kernels}
contains the proofs of \autoref{theo:symm-kernels} and \autoref{theo:nonlin}.
We treat the cases {\POne}, {\PTwo} resp.
{\QOne}, {\QTwo} in separate subsections because the product
structure of the kernels allows for a very short proof.
In \autoref{subsec:proof-nonlin} we explain how to prove \autoref{theo:nonlin}.
\autoref{sec:sym_rand} and \autoref{sec:nonsym} contain the proofs of
\autoref{theo:rand-kernels} and
\autoref{theo:nonsym} respectively.

\section{Symmetric resp. symmetrizable periodic
coefficients} \label{sec:proof-sym-kernels}

In this section we provide the proof of \autoref{theo:symm-kernels}. We provide
two different proofs, one for Case {\POne} and a separate one for Case
{\PTwo}. Both proofs can be adapted for the remaining case respectively but,
since the effective equation has a special
form under {\POne} and some proofs are shorter, we decide to look at
this case separately. Let us start with some general observations.

For $0 < \alpha < 2$ we consider  L\'{e}vy-type operators $\Lambda^\eps$ of
the form \eqref{eq:def-Leps}, where $\eps>0$ is a small positive
parameter. Our assumptions in Case {\POne} and Case {\PTwo} guarantee that
$\Lambda^\eps$ satisfies
\begin{align}\label{eq:elli}
\gamma^{-1}\leq \Lambda^\eps(x,y)\leq \gamma \qquad (x, y \in \R^d)
\end{align}
for some $\gamma > 1$ that does not depend on $\eps$. Condition \eqref{eq:elli}
can be seen as an ellipticity condition. As explained below, it guarantees
coercivity of the corresponding bilinear form in Sobolev spaces of fractional
order.

For each $\eps>0$ the
operator $L^\eps$ is symmetric on
$C_0^\infty(\mathbb R^d)$ in the weighted space $L^2(\mathbb R^d,\nu^\eps)$
where $\nu^\eps(x)=\nu(x/\eps)$  and $\nu(z)$ equals $\mu(z)/\lambda(z)$ in
Case {\POne} and $\nu(z)$ equals $1$ in Case {\PTwo}. Moreover, the
quadratic form $(-L^\eps u,v)_{L^2(\mathbb R^d,\nu^\eps)}$ is positive
on $C_0^\infty$. Indeed, in the case {\POne} for $u,\,v\in C_0^\infty(\mathbb
R^d)$ we have
\begin{align*}
(L^\eps u,v)_{L^2(\mathbb R^d,\nu^\eps)}&=\int\limits_{\mathbb
R^d}\int\limits_{\mathbb R^d}
\frac{\mu^\eps(x)\mu^\eps(y)}{|x-y|^{d+\alpha}}
\big(u(y)-u(x)\big)v(x)dydx \\
&=\int\limits_{\mathbb R^d}\int\limits_{\mathbb
R^d}
\frac{\mu^\eps(x)\mu^\eps(y)}{|x-y|^{d+\alpha}}
\big(u(y)v(x)-u(y)v(y)\big)dydx \\
&=\int\limits_{\mathbb
R^d}\int\limits_{\mathbb R^d}
\frac{\mu^\eps(x)\mu^\eps(y)}{|x-y|^{d+\alpha}}\big(v(x)-v(y)\big)u(y)dxdy=
(L^\eps v,u)_{L^2(\mathbb R^d,\nu^\eps)} \,.
\end{align*}
Here and in the sequel we denote $\mu^\eps(x)=\mu(x/\eps)$. In the
case {\PTwo} the symmetry can be checked in the same way.

The inequality $(L^\eps u,u)_{L^2(\mathbb R^d,\nu^\eps)}\leq 0$ follows from
the relation
$$
(L^\eps u,u)_{L^2(\mathbb R^d,\nu^\eps)}=-\frac12 \int\limits_{\mathbb
R^d}\int\limits_{\mathbb R^d}
\frac{\Lambda^\eps(x,y)}{|x-y|^{d+\alpha}}\big(u(y)-u(x)\big)^2\nu^\eps(x)dydx.
$$
The quadratic form  $(L^\eps u,v)_{L^2(\mathbb R^d,\nu^\eps)}$ with the domain
$H^{\alpha/2}(\mathbb R^d)$ is closed.
This follows from the fact that this quadratic form is comparable to
the quadratic form $(\Delta^{\alpha/2}u,v)$,
where $\Delta^{\alpha/2}=-(-\Delta)^{\alpha/2}$  is the fractional Laplacian.
The closedness of the last form is well-known. For the unique self-adjoint
operator
corresponding to this quadratic form (see \cite [Theorem X.23]{ReSi75}) we keep
the notation $L^\eps$, its domain is denoted $\mathcal{D}(L^\eps)$.   This
operator is self-adjoint and negative in the weighted space $L^2(\mathbb
R^d,\nu^\eps)$.

For a given constant
$m>0$ consider the resolvent $(m-L^\eps)^{-1}$. Since
$L^\eps$ is negative and self-adjoint in $L^2(\mathbb R^d,\nu^\eps)$, we have

\begin{align*}
\|(m-L^\eps)^{-1}\|_{\mathcal{L}(L^2(\mathbb R^d,\nu^\eps),L^2(\mathbb
R^d,\nu^\eps))}\leq \frac1m.
\end{align*}

In view of the properties of $\lambda$ and $\mu$ this yields

\begin{align}\label{eq:oc_res}
\|(m-L^\eps)^{-1}\|_{\mathcal{L}(L^2(\mathbb R^d),L^2(\mathbb R^d))}\leq
\frac{\gamma^2}m.
\end{align}

For a given $f\in L^2(\mathbb R^d)$ consider a sequence $(u^\eps)$ of
solutions to  equation \eqref{eq:ori_eqn}. Due to \eqref{eq:oc_res} for each
$\eps>0$ this equation has a unique solution, moreover $\|u^\eps\|_{L^2(\mathbb
R^s)} \leq \frac{\gamma^2}m \|f\|_{L^2(\mathbb R^s)}$.

As mentioned above, we provide
two proofs of \autoref{theo:symm-kernels}. In \autoref{subsec:proof-gamma-conv}
we provide a proof based on $\Gamma$-convergence. This proof is carried out
assuming {\POne}. Second, we assume {\PTwo} and prove
\autoref{theo:symm-kernels}
using compactness arguments in \autoref{subsec:2ndproof}. Note that either
proof works well in any of our
cases.

\subsection{First proof of \autoref{theo:symm-kernels}}
\label{subsec:proof-gamma-conv}

Assuming {\POne} we provide a proof of the Theorem based on
$\Gamma$-convergence. Consider
the functional
$$
F^\eps(u)=-(L^\eps u,u)_{L^2(\mathbb R^d,\nu^\eps)}+m(u,u)_{L^2(\mathbb
R^d,\nu^\eps)}-2
(f,u)_{L^2(\mathbb R^d,\nu^\eps)}
$$
for $u\in H^{\alpha/2}(\mathbb R^d)$. We extend this functional to the whole
$L^2(\mathbb R^d)$
letting $F^\eps(u)=+\infty$ for $u\in L^2(\mathbb R^d)\setminus
H^{\alpha/2}(\mathbb R^d)$.

It is straightforward to check that for each $\eps>0$ the functional $F^\eps$
is
continuous on
$H^{\alpha/2}(\mathbb R^d)$ and strictly convex. Thus, it attains its
minimum at a unique point. We denote this point by $u^\eps$.  It is
straightforward to see that $u^\eps$ belongs to $\mathcal{D}(L^\eps)$
and that $u^\eps$ is a solution of equation \eqref{eq:ori_eqn}.

We denote $L^2_w(\mathbb R^d)$ the space of square integrable functions
equipped
with the topology
of weak convergence. Here is our main auxiliary result.

\begin{theorem}\label{theo:gamma}
The family of functionals $F^\eps$ $\Gamma$-converges with respect to the
$L_{\rm loc}^2(\mathbb R^d)
\cap L^2_w(\mathbb R^d)$ topology
to the functional defined by
$$
F^{\rm eff}(u)=\frac12\int\limits_{\mathbb R^d}\int\limits_{\mathbb
R^d}\bar\mu^2\frac{(u(y)-u(x))^2}{|x-y|^{d+\alpha}}dydx+
\overline{\mu/\lambda}\int\limits_{\mathbb R^d}\{m (u(x))^2-2f(x)u(x)\}dx,
$$
for $u\in H^{\alpha/2}(\mathbb R^d)$ and $F^{\rm eff}(u)=+\infty$ for $u\in
L^2(\mathbb R^d)\setminus H^{\alpha/2}(\mathbb R^2)$,
where
$$
\bar\mu= \int_{[0,1]^d} \mu(y)dy, \qquad \overline{\mu/\lambda}
= \int_{[0,1]^d} (\mu(y))/(\lambda(y))dy.
$$
\end{theorem}

\begin{proof}[Proof of \autoref{theo:gamma}]
We begin with the $\Gamma$-$\liminf$ inequality. Let $v\in H^{\alpha/2}(\mathbb
R^d)$ and assume that
a sequence $v^\eps\in L^2(\mathbb R^d)$ converges to $v$ in $L^2_{\rm
loc}(\mathbb R^d)
\cap  L^2_{w}(\mathbb R^d)$ topology.

Denote
$$
Q^\eps(v):=\left\{
\begin{array}{ll}
\displaystyle
\int\limits_{\mathbb R^d}\!\!\int\limits_{\mathbb
R^d}\mu^\eps(x)\mu^\eps(y)\frac{(v(y)-v(x))^2}{|x-y|^{d+\alpha}}dydx,\quad&
v\in H^{\alpha/2}(\mathbb R^d);\\[6mm]
+\infty, \quad&v\in L^2(\mathbb R^d)\setminus
H^{\alpha/2}(\mathbb R^d);
\end{array}
\right.
$$
 and
$$
Q^0(v):=\left\{
\begin{array}{ll}
\displaystyle
\int\limits_{\mathbb R^d}\!\!\int\limits_{\mathbb
R^d}\bar\mu^2\frac{(v(y)-v(x))^2}{|x-y|^{d+\alpha}}dydx,\quad&
v\in H^{\alpha/2}(\mathbb R^d);\\[6mm]
+\infty, \quad&v\in L^2(\mathbb R^d)\setminus
H^{\alpha/2}(\mathbb R^d).
\end{array}
\right.
$$
From the definition of $F^\eps$ and $F^0$ it easily follows that
$$
\begin{array}{c}
c(m,f)\big(\|v\|\big._{H^{\alpha/2}(\mathbb R^d)}-1\big)\leq F^\eps(v)
\leq C(m,f)\big(\|v\|\big._{H^{\alpha/2}(\mathbb R^d)}+1\big),\\
c(m,f)\big(\|v\|\big._{H^{\alpha/2}(\mathbb R^d)}-1\big)\leq F^0(v)
\leq C(m,f)\big(\|v\|\big._{H^{\alpha/2}(\mathbb R^d)}+1\big)
\end{array}
$$
with strictly positive constants $c(m,f)$ and $C(m,f)$ that do not depend on
$\eps$.

Assume first that $F^0(v)=+\infty$. Then the $\Gamma$-$\liminf$ inequality is
trivial.
Indeed,  in this case $\|v\|_{H^{\alpha/2}(\mathbb R^d)}=+\infty$, and,
therefore,
$\liminf\limits_{\eps\to0}\|v^\eps\|_{H^{\alpha/2}(\mathbb R^d)}=+\infty$ for
any sequence
$v^\eps\in L^2(\mathbb R^d)$ that converges to $v$ in $L^2_{\rm loc}(\mathbb
R^d)$.
This yields the desired $\Gamma$-$\liminf$ inequality.
Assume now that $F^0(v)< +\infty$. It is clear that
\begin{align}\label{upp0}
\begin{array}{r}\displaystyle
\liminf\limits_{\eps\to0}\{(v^\eps,v^\eps)_{L^2(\mathbb
R^d,\nu^\eps)}-2(v^\eps,f)_{L^2(\mathbb R^d,\nu^\eps)}\}\\
\displaystyle
\geq \overline{\mu/\lambda}\big((v,v)_{L^2(\mathbb R^d)}-2(v,f)_{L^2(\mathbb
R^d)}\big)
\end{array}
\end{align}
for any sequence $v^\eps\in L^2(\mathbb R^d)$ that converges to $v$ in $L^2_
{\rm loc}(\mathbb R^d)
\cap L^2_w(\mathbb R^d)$
topology. Therefore, it suffices to show that
\begin{align}\label{qepsconv}
 \liminf\limits_{\eps\to0}Q^\eps(v^\eps)\geq Q^0(v).
 \end{align}
To this end we divide the integration area into three subsets as follows
\begin{align}\label{divi_spa}
\mathbb R^d\times\mathbb R^d=G^\delta_1\cup G^\delta_2\cup G^\delta_3
\end{align}
with
\begin{align}\label{defn_g1}
G^\delta_1=\{(x,y)\,:\,|x-y|\geq\delta,\,|x|+|y|\leq \delta^{-1}\},
\end{align}
\begin{align}\label{defn_g23}
G^\delta_2=\{(x,y)\,:\,|x-y|\leq\delta,\,|x|+|y|\leq \delta^{-1}\},\quad
G^\delta_3=\{(x,y)\,:\,|x|+|y|\geq \delta^{-1}\}.
\end{align}
Since the integral
$$
\int\limits_{\mathbb R^d\times\mathbb R^d}
\frac{\big(v(y)-v(x)\big)^2}{|x-y|^{d+\alpha}}dydx
$$
converges, for any $\kappa>0$ there exists $\delta>0$ such that
\begin{align}\label{upp1}
\int\limits_{G^\delta_2\cup G^\delta_3}
\overline{\mu}^2 \frac{\big(v(y)-v(x)\big)^2}{|x-y|^{d+\alpha}}dydx\leq \kappa
\end{align}
Obviously,
\begin{align}\label{upp2}
\liminf\limits_{\eps\to0}\int\limits_{G^\delta_2\cup G^\delta_3}
\mu^\eps(y)\mu^\eps(x) \frac{\big(v(y)-v(x)\big)^2}{|x-y|^{d+\alpha}}dydx\geq 0.
\end{align}
In the domain $G^\delta_1$ we have
$$
0< c_1(\delta)\leq \frac{\mu^\eps(y)\mu^\eps(x)}{|x-y|^{d+\alpha}}\leq
C_1(\delta),
$$
and $v^\eps$ converges to $v$ in $L^2(G_1^\delta)$. Therefore, as $\eps\to0$,
$$
\int\limits_{G^\delta_1}\mu^\eps(y)\mu^\eps(x)\frac{(v^\eps(x))^2}{|x-y|^{
d+\alpha}}dydx
\longrightarrow
\int\limits_{G^\delta_1}\bar\mu^2\frac{(v(x))^2}{|x-y|^{d+\alpha}}dydx,
$$
and
$$
\int\limits_{G^\delta_1}\mu^\eps(y)\mu^\eps(x)\frac{v^\eps(y)v^\eps(x)}{|x-y|^{
d+\alpha}}dydx
\longrightarrow
\int\limits_{G^\delta_1}\bar\mu^2\frac{v(y)v(x)}{|x-y|^{d+\alpha}}dydx.
$$
This yields
\begin{align}\label{upp3}
\lim\limits_{\eps\to0}\int\limits_{G^\delta_1}
\mu^\eps(y)\mu^\eps(x)\frac{(v^\eps(y)-v^\eps(x))^2}{|x-y|^{d+\alpha}}dydx
=
\int\limits_{G^\delta_1}\bar\mu^2\frac{(v(y)-v(x))^2}{|x-y|^{d+\alpha}}dydx.
\end{align}
Combining \eqref{upp0}--\eqref{upp3} we conclude that
$$
\liminf\limits_{\eps\to0} F^\eps(v^\eps)\geq F^0(v)-\kappa.
$$
Since $\kappa$ is an arbitrary positive number, the desired $\Gamma$-$\liminf$
inequality
follows.

We turn to the $\Gamma$-$\limsup$ inequality. It suffices to set $v^\eps=v$. It
is straightforward to check that $F^\eps(v)\to F^0(v)$. This completes the
proof
of \autoref{theo:gamma}.
\end{proof}


We can finally provide the proof of our main result in Case {\POne}.

\begin{proof}[Proof of \autoref{theo:symm-kernels}]
As a consequence of \autoref{theo:gamma} any limit point of $\{u^\eps\}$ is a
minimizer of $F^0$,
see  \cite[Theorem 1.21]{braides2005gamma}.
Since the minimizer of $F^0$ is unique, the whole family $\{u^\eps\}$
converges,
as $\eps\to0$, to
$u=\mathrm{argmin}F^0$ in $L^2_w(\mathbb R^d)\cap L^2_{\rm loc}(\mathbb R^d)$
topology; here the subindex $w$ indicates the weak topology.

It remains to show that $u^\eps$ converges to $u$ in $L^2(\mathbb R^d)$. If we
assume that $u^\eps$ does not converge to $u$ in $L^2(\mathbb R^d)$, then, for
a
subsequence, for any $n\in\mathbb Z$ there exists $\eps(n)>0$ such that for any
$\eps<\eps(n)$ we have
$$
\|u^\eps\|_{L^2(\mathbb R^d\setminus G(n))}\geq C_2,
$$
where $C_2>0$ is a constant that does not depend on $n$, and $G(n)$ stands for
the ball of radius $n$
centered at the origin. For sufficiently small $\eps$ this inequality
contradicts the fact that $u^\eps$ is a minimizer of $F^\eps$. Thus $u^\eps$
converges in $L^2 (\mathbb R^d)$.

The minimizer $u$ satisfies the equation
$$
\bar\mu^2\Delta^{\alpha/2}u-\overline{\mu/\lambda}mu=\overline{\mu/\lambda}f.
$$
Dividing it by $\overline{\mu/\lambda}$ we arrive at \eqref{eq:lim_eqn}.
\autoref{theo:symm-kernels} is proved. \end{proof}

\subsection{Second proof of \autoref{theo:symm-kernels}} \label{subsec:2ndproof}

In this section we give the second proof of   \autoref{theo:symm-kernels}. Here
we assume that condition {\PTwo} holds. This proof can be easily adapted to the
Case {\POne}.

\begin{proof}[Second proof of \autoref{theo:symm-kernels}]
Here we consider an operator $L^\eps$ of the form
\begin{align}\label{operat_fullsym}
L^\eps u(x)=\int\limits_{\mathbb R^d} \frac{\displaystyle
\Lambda\left(x,y,\frac
x\eps,\frac y\eps\right)
\big(u(y)-u(x)\big)}{|y-x|^{d+\alpha}}\:dy
\end{align}
with a continuous in $(x,y)$ and periodic measurable in $\zeta$ and $\eta$
function $\Lambda(x,y,\zeta,\eta)$ such that
$$
\Lambda(x,y,\zeta,\eta)=\Lambda(y,x,\eta,\zeta),\qquad \gamma^{-1}\leq
\Lambda(x,y,\zeta,\eta)
\leq\gamma.
$$
Our assumptions on the setup ensure that $\Lambda$ is a Carath\'{e}odory
function and $\Lambda^\eps$ is well defined.


As was explained above, $L^\eps$ is
a positive self-adjoint operator in
$L^2(\mathbb R^d)$
whose domain $D(L¨^ \eps)$ belongs to $H^{\alpha/2}(\mathbb R^d)$.

Multiplying the equation $-L^\eps u^\eps+ m u^\eps=f$ by $u^\eps$ and
integrating the resulting
relation over $\mathbb R^d$ we conclude
$$
\|u^\eps\|_{H^{\alpha/2}(\mathbb R^d)}\leq C
$$
with a constant $C$ that does not depend on $\eps$. Therefore, for a
subsequence,
$u^\eps$ converges to some function $u\in H^{\alpha/2}(\mathbb R^d)$, weakly in
$H^{\alpha/2}(\mathbb R^d)$
and strongly in $L_{\rm loc}^2(\mathbb R^d)$. In order to characterize
this limit function we multiply the equation  $-L^\eps u^\eps+ m u^\eps=f$ by a
test function
$\varphi\in C_0^\infty(\mathbb R^d)$ and integrate the obtained relation over
$\mathbb R^d$. After
simple rearrangements this yields
$$
\int\limits_{\mathbb R^d\times\mathbb
R^d}\frac{\Lambda^\eps(x,y)(u^\eps(y)-u^\eps(x))(\varphi(y)-\varphi(x))}{|x-y|^{
d+\alpha}}\,dxdy
+\int\limits_{\mathbb R^d}(u^\eps\varphi-f\varphi)dx=0,
$$
where $\Lambda^\eps(x,y)$ stands for $\Lambda\big(x,y,\frac x\eps,\frac
y\eps\big)$.
Clearly, the second integral converges to the integral $\int\limits_{\mathbb
R^d}(u\varphi-f\varphi)dx$.
Our goal is to pass to the limit in the first one. To this end we divide the
integration area $\mathbb R^d\times\mathbb R^d$ into three parts in the same
way
as it was done in \eqref{divi_spa}, \eqref{defn_g1} and \eqref{defn_g23}.
 The integral over
$G_2^\delta\cup G_3^\delta$ admits the following estimate
\begin{align*}
\bigg|\int\limits_{G_2^\delta\cup G_3^\delta} &
\frac{\Lambda^\eps(x,y)(u^\eps(y)-u^\eps(x))(\varphi(y)-\varphi(x))}{|x-y|^{
d+\alpha}}\,dxdy \bigg| \\
&\leq C\left(\int\limits_{G_2^\delta\cup
G_3^\delta}\frac{(u^\eps(y)-u^\eps(x))^2}
{|x-y|^{d+\alpha}}\,dxdy\right)^{\frac12}
\left(\int\limits_{G_2^\delta\cup G_3^\delta}\frac{(\varphi(y)-\varphi(x))^2}
{|x-y|^{d+\alpha}}\,dxdy\right)^{\frac12} \\
&\leq C_1\left(\int\limits_{G_2^\delta\cup
G_3^\delta}\frac{(\varphi(y)-\varphi(x))^2}
{|x-y|^{d+\alpha}}\,dxdy\right)^{\frac12}
\end{align*}
The last integral tends to zero, as $\delta\to 0$. Similarly,
$$
\bigg|\int\limits_{G_2^\delta\cup G_3^\delta}
\frac{\overline\Lambda(x,y)(u(y)-u(x))(\varphi(y)-\varphi(x))}{|x-y|^{
d+\alpha}}\,dxdy \bigg|
\,\longrightarrow\,0,
 $$
as $\delta\to0$. \\
According to \cite[Lemma 3.1]{Zhikov2004} the family $\Lambda^\eps$ converges
weakly in $L^2_{\rm loc}(\mathbb R^d\times\mathbb R^d)$
to the function  $\overline\Lambda$ with
$\overline\Lambda(x,y)=\Lambda^{\rm eff}(x,y)$.
Since
$u^\eps$ converges to $u$ in $L^2(G_1^\delta)$
and $\Lambda^\eps$ converges to $\overline\Lambda$ weakly on any bounded
domain, we conclude
\begin{align*}
\int\limits_{G_1^\delta} &
\frac{\Lambda^\eps(x,y)(u^\eps(y)-u^\eps(x))(\varphi(y)-\varphi(x))}{|x-y|^{
d+\alpha}}\,dxdy
\\
& \mathop{\longrightarrow}\limits_{\eps\to0}  \int\limits_{G_1^\delta}
\frac{\overline\Lambda(x,y)(u(y)-u(x))(\varphi(y)-\varphi(x))}{|x-y|^{d+\alpha}}
\,
dxdy \,.
\end{align*}

Combining the above relations, we arrive at the conclusion that
$$
\int\limits_{\mathbb R^d\times\mathbb R^d}
\frac{\overline\Lambda(x,y)(u(y)-u(x))(\varphi(y)-\varphi(x))}{|x-y|^{
d+\alpha}}\,dxdy
+\int\limits_{\mathbb R^d}(u\varphi-f\varphi)dx=0.
$$
Since $\varphi$ is an arbitrary $C_0^\infty$ function, this implies that $u$ is
a solution of the equation $-L^0u+mu=f$. Due to the uniqueness of a solution of
this equation, the whole family
$u^\eps$ converges to $u$, as $\eps\to0$.

It remains to justify the convergence in $L^2(\mathbb R^d)$. We have
$$
0\leq(-L^\eps(u^\eps-u),u^\eps-u)=-(L^\eps u^\eps,u^\eps)+2(L^\eps u^\eps,
u)-(L^\eps u,u).
$$
Passing to the limit yields
\begin{align}\label{bolshe}
\liminf\limits_{\eps\to0} \big\{-(L^\eps u^\eps,u^\eps)\big\}\geq -(L^0u,u).
\end{align}
Now the strong convergence of $u^\eps$ in $L^2(\mathbb R^d)$ can be obtained by
the standard lower semicontinuity arguments. Indeed, multiplying the equation
$-L^\eps u^\eps+ m u^\eps=f$
by $u^\eps$, integrating the resulting relation over $\mathbb R^d$ and passing
to the limit as $\eps\to0$ we have
$$
\lim\limits_{\eps\to0} \big((-L^\eps u^\eps,u^\eps)+ m
(u^\eps,u^\eps)\big)=(f,u).
 $$
If $u^\eps$ does not converge strongly in $L^2(\mathbb R^d)$ then for a
subsequence
$\lim\limits_{\eps\to0}m(u^\eps,u^\eps)>m(u,u)$. Combining this with
\eqref{bolshe}
for the same subsequence we obtain
 $$
\lim\limits_{\eps\to0} \big((-L^\eps u^\eps,u^\eps)+ m
(u^\eps,u^\eps)\big)>-(L^0u,u)+m(u,u)=(f,u).
 $$
The last relation here follows from the limit equation $-L^0 u+mu=f$. We arrive
at a contradiction. Thus
$u^\eps$ converges to $u$ in norm.
\end{proof}

\subsection{Proof of \autoref{theo:nonlin}} \label{subsec:proof-nonlin}

Let us comment on the proof of \autoref{theo:nonlin}. As mentioned above, the
proof does not require any new idea but just an adjustment of the setting. For
 every $\eps>0, m> 0$ the equation \eqref{eq:nonlin-eps-eq} posesses a
unique solution $u^\eps \in W^{\frac{\alpha}{p},p}(\mathbb
R^d)$. It minimizes the variational functional
\begin{align*}
 v \mapsto J(v) = & \frac1p \int\limits_{\mathbb R^d} \int\limits_{\mathbb
R^d}
\frac{|v(y)-v(x)|^p}{|x-y|^{d+
\alpha}}\Lambda^\eps(x,y)\,dydx + \frac{m}{p} |v|^{p} + \int\limits_{\R^d} f
v \,.
\end{align*}
In order to establish bounds that are uniform in $\eps$, we multiply
\eqref{eq:nonlin-eps-eq} by $u^\eps$, integrate the resulting relation over
$\mathbb R^d$, and exploit the equality
\begin{align*}
\int_{\mathbb R^d}\int_{\mathbb R^d}
\frac{|u(y)-u(x)|^{p-2}(u(y)-u(x)) u(x)}{|x-y|^{d+
\alpha}}\Lambda^\eps(x,y) \,dydx \\
=-\frac12\int_{\mathbb R^d}\int_{\mathbb R^d}
\frac{|u(y)-u(x)|^{p}}{|x-y|^{d+ \alpha}} \Lambda^\eps(x,y) \,dydx
\end{align*}
Then, we easily deduce the estimate
\begin{equation}\label{p-est}
  \|u^\eps\|_{W^{\frac{\alpha}{p},p}(\mathbb R^d)}\leq
C\|f\|_{L^{p'}(\mathbb R^d)}.
\end{equation}
with a constant $C$ that does not depend on $\eps$. Thus, there is a weakly
convergent subsequence and a limit $u^0$. From here, the proof is the same as that
of \autoref{theo:symm-kernels}.

\section{Symmetric random kernels} \label{sec:sym_rand}

Let us first explain the notion of a ergodic dynamical system. Let
$(\Omega,\mathcal{F},\mathbf{P})$ be a standard probability space and
assume that $(T_y)_{y\in \mathbb R^d}$,  is a $d$-dimensional ergodic
dynamical
system in this probability space, i.e., a collection of measurable maps
$T_y\,:\,\Omega
\mapsto\Omega$ such that
\begin{itemize}
  \item $T_{y_1}T_{y_2}=T_{y_1+y_2}$ for all $y_1$ and $y_2$ in $\mathbb R^d$;
\ $T_0=\mathrm{Id}$;
  \item $\mathbf{P}(T_y A)=\mathbf{P}(A)$ for all $A\in\mathcal{F}$ and all
$y\in\mathbb R^d$;
  \item $T_\cdot\,:\,\mathbb R^d\times\Omega\mapsto\Omega$ is a measurable map.
Here $\mathbb R^d\times\Omega$
  is equipped with the $\sigma$-algebra $\mathcal{B}\times\mathcal{F}$, where
$\mathcal{B}$ is the Borel $\sigma$-algebra
  in $\mathbb R^d$.
\end{itemize}
We say that $T_y$ is ergodic if for any $A\in\mathcal{F}$ such that $T_y A=A$
for all $y\in\mathbb R^d$ we have either
$\mathbf{P}(A)=0$ or $\mathbf{P}(A)=1$.

\medskip

Let us first make some remarks. We study the limit behaviour of operator
$L^\eps$ defined in \eqref{eq:def-Leps}, as
$\eps\to0$. Clearly, estimate \eqref{eq:oc_res} remains valid in the random
case. Therefore, for any given $f\in L^2(\mathbb R^d)$ the sequence of equations
\begin{align}\label{ori_eqn_rand}
(L^\eps-m)u^\eps=f
\end{align}
is well-posed. Moreover,  for any $\eps>0$ a solution $u^\eps$ is uniquely
defined, and   $\|u^\eps\|_{L^2(\mathbb R^s)}
\leq \frac{\gamma^2}m \|f\|_{L^2(\mathbb R^s)}$.

\subsection{First proof of
\autoref{theo:rand-kernels}} \label{subsec:proof-one-rand-kernels}

Now we are in the position to prove \autoref{theo:rand-kernels} in Case {\QOne}.

\begin{proof}[Proof of \autoref{theo:rand-kernels} in Case {\QOne}]
In the same way as in the proof of \autoref{theo:symm-kernels} for any $f\in
L^2(\mathbb
R^d)$ we obtain the estimate
$$
\|u^\eps\|_{H^{\alpha/2}(\mathbb R^d)}\leq C
$$
with a deterministic constant $C$ that does not depend on $\eps$.  Therefore,
for each $\omega\in\Omega$ there is a subsequence
that converges to a function $u^0\in H^{\alpha/2}(\mathbb R^d)$ weakly in
$H^{\alpha/2}(\mathbb R^d)$ and strongly
in $L^2_{\rm loc}(\mathbb R^d)$. Abusing slightly the notation we keep for this
subsequence the same name $u^\eps$.

Multiplying equation \eqref{ori_eqn_rand} by $\mu(\frac x\eps)(\lambda(\frac
x\eps))^{-1}\varphi(x)$ with
$\varphi\in C_0^\infty(\mathbb R^d)$ and integrating the resulting equality
over
$\mathbb R^d$ after simple rearrangements we arrive
 at the following relation:

\begin{align*}
0 = \int\limits_{\mathbb R^d\times\mathbb
R^d} &\!\!\!\!\frac{
\mu^\eps(y)\mu^\eps(x)(u^\eps(y)-u^\eps(x))(\varphi(y)-\varphi(x))}{|x-y|^{
d+\alpha}}\,dxdy \\
& + \int\limits_{\mathbb
R^d}\!\frac{\mu^\eps(x)}{\lambda^\eps(x)}(u^\eps\varphi-f\varphi)dx\,.
\end{align*}

Here, $\mu^\eps(x)$ and $\lambda^\eps(x)$ stand for  $\mu(\frac x\eps)$ and
$\lambda(\frac x\eps)$, respectively. By the Birkhoff ergodic theorem
$\mu^\eps(y)\mu^\eps(x)$ converges a.s., as $\eps\to0$,  to
$\big(\mathbf{E}\{\widehat\mu(\cdot)\}\big)^2$ weakly in
$L^2_{\rm loc}(\mathbb R^d\times\mathbb R^d)$.  Similarly,
$\frac{\mu^\eps(x)}{\lambda^\eps(x)}$ converges a.s. to
$\mathbf{E}\big\{\frac{\widehat\mu(\cdot)}{\widehat\lambda(\cdot)}\big\}$ in
$L^2_{\rm
loc}(\mathbb R^d)$.

Following the line of the second proof of \autoref{theo:symm-kernels} we obtain
\begin{align*}
0= \int\limits_{\mathbb R^d\times\mathbb
R^d} &\!\!\!\!\frac{\big(\mathbf{E}\{\widehat\mu(\cdot)\}
\big)^2(u^0(y)-u^0(x))(\varphi(y)-\varphi(x))}{|x-y|^{d+\alpha}}\,dxdy \\
& +\mathbf{E}\Big\{\frac{\widehat\mu(\cdot)}{\widehat\lambda(\cdot)}\Big\}
\int\limits_ {
\mathbb R^d}\!(u^0\varphi-f\varphi)dx \,.
\end{align*}

This yields the desired relation \eqref{eq:lim_eqn}.  The fact that the whole
family $\{u^\eps\}$ converges to $u^0$ a.s.
follows from the uniqueness of a solution of equation \eqref{eq:lim_eqn}.
Finally, the convergence
$\lim\limits_{\eps\to0}\|u^\eps-u^0\|_{L^2(\mathbb R^d)}=0$ can be justified in
the same way as in the second proof of \autoref{theo:symm-kernels}.
\end{proof}

\subsection{Second proof of
\autoref{theo:rand-kernels}} \label{subsec:proof-two-rand-kernels}

Next, we explain how to establish \autoref{theo:rand-kernels} in Case {\QTwo}.
The proof will follow in a straightforward way once we have established the
following auxiliary result.

\begin{lemma}\label{lem:mv}
For any bounded Lipschitz domain $Q\subset
\mathbb R^d\times\mathbb R^d$ we have a.s.
\begin{equation}\label{f_mv}
\int_Q\Lambda(x,y,T_{\frac x\eps}\omega,T_{\frac
y\eps}\omega)\,dxdy\longrightarrow
\int_Q\Lambda^{\rm eff}(x,y)\,dxdy, \qquad\hbox{as }\eps\to0,
\end{equation}
where
$$
\Lambda^{\rm eff}(x,y)=\int_\Omega\int_\Omega
\Lambda(x,y,\omega_1,\omega_2)d\mathbf{P}(\omega_1)d\mathbf{P}(\omega_2).
$$
\end{lemma}

\begin{proof}
Notice first that under the assumptions of the lemma the function $\Lambda^{\rm
eff}(x,y)$ is continuous
on $\overline Q$.\\
Since $\overline Q\times\Omega\times\Omega$ is compact, for any $\delta>0$ there
exists $\varkappa>0$
such that
$$
|\Lambda(x',y',\omega_1,\omega_2)-\Lambda(x'',y'',\omega_1,\omega_2)|\leq\delta
\quad\hbox{for all }\omega_1,\ \omega_2,
$$
if $|(x',y')-(x'',y'')|\leq\varkappa$.  \\
Consider a partition $\{B_j\}_{j=1}^{N(\delta)}$ of $Q$ that has the following
properties:
\begin{itemize}
\item[(i)]  $\overline Q=\bigcup\overline{B}_j$,\ \ \ \ $B_j\cap B_k=\emptyset$
if $j\not=k$.
\item[(ii)] $\mathrm{diam}(B_j)\leq \varkappa$.
\item[(iii)] The inequality holds
$$
\Big|\int_Q\Lambda^{\rm eff}(x,y)\,dxdy-\sum\limits_{j=1}^N \Lambda^{\rm
eff}(x_j,y_j)|B_j|\Big|\leq\delta
$$
where $\{(x_j,y_j)\}_{j=1}^N$ is a set of points in $Q$ such that $(x_j,y_j)\in
B_j$.
\end{itemize}
By the Stone-Weierstrass theorem for each $j=1,\ldots,N$ there exist a finite
set of continuous functions
$\{\varphi_k^j(\omega),\,\psi_k^j(\omega)\}_{k=1}^L$ such that
$$
\Big|\Lambda(x_j,y_j,\omega_1,\omega_2)-\sum\limits_{k=1}
^L\varphi_k^j(\omega_1)\psi_k^j(\omega_2)\Big|\leq
\delta.
$$
This implies in particular that
\begin{equation}\label{f_ave}
\Big|\Lambda^{\rm
eff}(x_j,y_j)-\sum\limits_{k=1}^L\mathbf{E}\varphi_k^j\mathbf{E}
\psi_k^j\Big|\leq\delta.
\end{equation}
Then we have a.s.

\begin{align*}
\limsup\limits_{\eps\to0} & \int\limits_Q\Lambda(x,y,T_{\frac
x\eps}\omega,T_{\frac y\eps}\omega)\,dxdy \\
& \leq\limsup\limits_{\eps\to0}\sum_{j=1}^N\int\limits_{B_j}
\Lambda(x_j,y_j,T_{\frac x\eps}\omega,T_{\frac y\eps}\omega)\,dxdy+\delta|Q| \\
&\leq\limsup\limits_{\eps\to0}\sum_{j=1}^N\int\limits_{B_j}\sum\limits_{k=1}^L
\varphi_k^j(T_{\frac x\eps}\omega)\psi_k^j(T_{\frac
y\eps}\omega)\,dxdy+2\delta|Q| \\
&= \sum_{j=1}^N\int\limits_{B_j}\sum\limits_{k=1}^L\mathbf{E}\varphi_k^j\mathbf{
E } \psi_k^j\,dxdy+2\delta|Q| \\
& \leq\sum_{j=1}^N\int\limits_{B_j}
\Lambda^{\rm eff}(x_j,y_j)\,dxdy+3\delta|Q|\leq \int\limits_{Q}
\Lambda^{\rm eff}(x,y)\,dxdy+\delta(3|Q|+1)\,.
\end{align*}

The third relation here follows from the Birkhoff ergodic theorem, and the
fourth one from estimate \eqref{f_ave}. Similarly,
$$
\liminf\limits_{\eps\to0}\int\limits_Q\Lambda(x,y,T_{\frac x\eps}\omega,T_{\frac
y\eps}\omega)\,dxdy
\geq \int\limits_{Q}\Lambda^{\rm eff}(x,y)\,dxdy-\delta(3|Q|+1).
$$
Since $\delta > 0$ can be chosen arbitrarily, this implies the desired relation
\eqref{f_mv}.
\end{proof}

With the help of this convergence result, the proof of
\autoref{theo:rand-kernels} is immediate.

\section{Non-symmetric kernels} \label{sec:nonsym}

The aim of this section of to prove \autoref{theo:nonsym}. We split the proof
into three different steps. In \autoref{subsec:aux_nonsymm} we investigate the
adjoint operator $L^*$ and its principal eigenfunction.
\autoref{subsec:apriori_nonsymm} provides uniform bounds on the functions
$u^\eps$. Finally, we consider the limit $\eps \to 0$ in
\autoref{subsec:limit_nonsymm}.

\subsection{Auxiliary periodic problems} \label{subsec:aux_nonsymm}


Without loss of generality we suppose that the period of $\Lambda$ in each
variable is $[0,1]^d$.
We deal here with an auxiliary (cell) problem defined in  the space of periodic
functions $L^2(\T^d)$. Notice that in this case the operator
$$
Lv(\zeta)=\int\limits_{\mathbb
R^d}\frac{\Lambda(\zeta,\eta)\big(v(\eta)-v(\zeta)\big)}
{|\zeta-\eta|^{d+\alpha}}\,d\eta
$$
is an unbounded linear operator in $L^2(\T^d)$; here and in what follows we identify periodic functions defined on the torus
$\mathbb T^d$ with the corresponding periodic functions in $\mathbb R^d$.  With a domain
$\mathcal{D}(L)=H^\alpha(\T^d)$
this operator is closed and its adjoint also has a domain $H^\alpha(\mathbb
T^d)$.
Direct computations show that the adjoint operator takes the form
$$
L^*q(\zeta)=\int\limits_{\mathbb
R^d}\frac{\big(\Lambda(\eta,\zeta)q(\eta)-\Lambda(\zeta,\eta)q(\zeta)\big)}
{|\zeta-\eta|^{d+\alpha}}\,d\eta.
$$

\begin{theorem}\label{theo:nonsym-eigenvalue}
The kernel of operator $L^*$ in $L^2(\T^d)$ has dimension one. The
corresponding eigenfunction
$p_0(\xi)$ is continuous and, under proper normalization, positive. Moreover,
there exists a constant $p^->0$
such that $p_0(\xi)\geq p^-$ for all $\xi\in\T^d$.
\end{theorem}

The remainder of this subsection is dedicated to the proof of
\autoref{theo:nonsym-eigenvalue}, which itself uses several auxiliary results.

\begin{proof}[Proof of \autoref{theo:nonsym-eigenvalue}]
First we are going to show that the kernel of $L^*$ in $L^2(\T^d)$
contains a continuous positive function,
we denote it $p_0$. The uniqueness will be justified later on.
To prove the existence of such a function  $p_0$ we check that the Krein-Rutman
theorem applies to the resolvent
of the operators $L$ and $L^*$.

We represent the operator $L^*$ in the form
$$
-L^*q(\zeta)= \Lambda(\zeta,\zeta)\int\limits_{\mathbb
R^d}\frac{\big(q(\zeta)-q(\eta)\big)}
{|\zeta-\eta|^{d+\alpha}}\,d\eta+\int\limits_{\mathbb
R^d}\frac{\big(\Lambda(\zeta,\zeta)-\Lambda(\eta,\zeta)\big)q(\eta)}
{|\zeta-\eta|^{d+\alpha}}\,d\eta
$$
$$
+q(\zeta)\int\limits_{\mathbb
R^d}\frac{\big(\Lambda(\zeta,\zeta)-\Lambda(\zeta,\eta)\big)}
{|\zeta-\eta|^{d+\alpha}}\,d\eta =:
\Lambda(\zeta,\zeta)\big[\mathcal{L}^sq(\zeta)+ \mathcal{L}^1q(\zeta)+
\mathcal{L}^2q(\zeta)\big]
$$
Since $\Lambda(\zeta,\eta)$ is a Lipschitz continuous function and $\alpha\in(0,1)$, the kernel of
the operator  $\mathcal{L}^1$  is integrable on
$\mathbb R^d\times \T^d$. Considering the fact that this kernel is
continuous on the complement of the set $\{(\zeta,\eta)\,:\,\zeta=\eta\}$,
we conclude that $\mathcal{L}^1$ is a bounded operator in $C(\T^d)$.  The
function
$\int\limits_{\mathbb R^d}\frac{(\Lambda(\eta,\zeta)-\Lambda(\zeta,\zeta))}
{\Lambda(\zeta,\zeta)|\zeta-\eta|^{d+\alpha}}\,d\eta$ is  continuous  and
periodic. Therefore, the operator   $\mathcal{L}^2$ is also bounded in
$C(\T^d)$.

\begin{lemma}\label{l_hold}
There exists $\beta>0$ such that for any $\lambda>0$ the resolvent
$(\mathcal{L}^s+ \lambda\mathbf{I})^{-1}$ is a bounded operator from $C(\mathbb
T^d)$ to $C^\beta(\T^d)$. Moreover,
the following estimate holds
$$
\|(\mathcal{L}^s+ \lambda\mathbf{I})^{-1}\|\big._{C(\T^d)\to C(\mathbb
T^d)}\leq \lambda^{-1}.
$$
\end{lemma}

The second statement follows directly from the maximum principle. We
reformulate the first statement as a separate result.

\begin{proposition}
Let $f \in C(\T^d), \lambda > 0$. There are constants $\delta > 0$, $c \geq 1$
such that for every function $u \in H^\alpha(\T^d)$ satisfying
\begin{align}\label{eq:resol-eq}
 (-\Delta)^\frac{\alpha}{2} u + \lambda u = f \quad \text{ in } \mathbb T^d
\end{align}
the following estimate holds:
\begin{align}\label{eq:holder-estim}
\|u\|_{C^\delta(\T)} \leq c  \|f\|_{C(\T^d)}
\end{align}
\end{proposition}

\begin{proof}
There are several ways to prove this result. One option would be to apply
embedding results for the Riesz potential. Another option would be to use
the Harnack inequality. Here, we give a proof based on the corresponding heat
equation and the representation of solutions with the help
of the fundamental solution. Let $(P_t)$ denote the contraction semigroup of
the operator $\partial_t +
(-\Delta)^\frac{\alpha}{2}$ in $(0,\infty)\times\mathbb R^d$.
It is known that for  $f\in L^\infty(\mathbb R^d)$ the
function $P_t f$ belongs to $C^\infty(\mathbb R^d)$ and satisfies
\begin{align}\label{eq:estim-heat-semigroup}
|\nabla P_t f (x)| \leq c_1 t^{-1/\alpha} \|f\|_\infty \qquad \hbox{for all }x \in \R^d
\end{align}
with some contant $c_1 \geq 1$ independent of $x$. This is proved in several
works, e.g. in \cite[Theorem 3.2]{MR2555009}. In order to prove
\eqref{eq:holder-estim}, let $u$ be a solution to \eqref{eq:resol-eq} and $x,y
\in \R^d$. We only need to consider the case $|x-y| \leq 1$. Assume $\rho \in
(0,1)$. Then
\begin{align*}
| u(x) - u(y) | &\leq \int\limits_0^\infty e^{- \lambda t} |P_t f(x) - P_t
f(y)| dt \\
&\leq \int\limits_0^\rho e^{- \lambda t}
|P_t f(x) - P_t f(y)| dt + \int\limits_\rho^\infty e^{- \lambda t}
|P_t f(x) - P_t f(y)| dt \,.
\end{align*}
The first integral is estimated from above as follows:
\begin{align*}
\int\limits_0^\rho e^{- \lambda t} |P_t f(x) - P_t f(y)| dt &\leq 2
\|f\|_\infty
\int\limits_0^\rho e^{- \lambda t} dt \\
&= 2 \|f\|_\infty \frac{1}{\lambda} ( 1-
e^{-\lambda
\rho}) \leq 2 \rho\|f\|_\infty  \,.
\end{align*}
For the estimate of the second integral we apply
\eqref{eq:estim-heat-semigroup} and obtain
\begin{align*}
\int\limits_\rho^\infty e^{- \lambda t} |P_t f(x) - P_t f(y)| dt \leq c_1
\|f\|_\infty |x-y| \int\limits_\rho^\infty e^{- \lambda t} t^{-1/\alpha} dt
\end{align*}
Note that for $\alpha < 1$  we have
\begin{align*}
\int\limits_\rho^\infty e^{- \lambda t} t^{-1/\alpha} dt &\leq
\int\limits_\rho^1 t^{-1/\alpha} dt + \int\limits_1^\infty e^{- \lambda t}  dt
\leq \frac{\alpha-1}{\alpha} (1- \rho^{\frac{\alpha-1}{\alpha}}) +
\frac{1}{\lambda} e^{-\lambda} \\
&\leq c_2(\alpha) \max\{1, \rho^{\frac{\alpha-1}{\alpha}}\} + c_3(\lambda) \,.
\end{align*}
Hence, we obtain for $\alpha< 1$
\begin{align*}
\int\limits_\rho^\infty e^{- \lambda t} |P_t f(x) - P_t f(y)| dt \leq c_1 |x-y|
\|f\|_\infty \Big( c_2 \max\{1, \rho^{\frac{\alpha-1}{\alpha}}\} +
c_3 \rho \Big)
\end{align*}
Now we choose $\rho=|x-y|^\alpha$. Combining the estimates of the two
integrals, we obtain the desired result with $\delta=\alpha$.
\end{proof}

\pagebreak[3]

\begin{lemma}\label{l_ctoc}
There exist $\lambda_0>0$ and $\beta>0$ such that for all $\lambda\geq
\lambda_0$ the resolvent
$(\mathcal{L}^s+ \mathcal{L}^1+ \mathcal{L}^2 +\lambda\mathbf{I})^{-1}$ is a
bounded operator from $C(\T^d)$ to $C^\beta(\T^d)$.
\end{lemma}
\begin{proof}
We have
$$
(\mathcal{L}^s+ \mathcal{L}^1+ \mathcal{L}^2 +\lambda\mathbf{I})^{-1}=
\big([\mathbf{I}+( \mathcal{L}^1+
\mathcal{L}^2)(\mathcal{L}^s+\lambda\mathbf{I})^{-1}](\mathcal{L}
^s+\lambda\mathbf{I})\big)^{-1}
$$
$$
=(\mathcal{L}^s+\lambda\mathbf{I})^{-1}[\mathbf{I}+( \mathcal{L}^1+
\mathcal{L}^2)(\mathcal{L}^s+\lambda\mathbf{I})^{-1}]^{-1}
$$
Letting $\lambda_0=2\|\mathcal{L}^1+\mathcal{L}^2\|\big._{C(\T^d)\to
C(\T^d)}$   one can easily check that
$[\mathbf{I}+( \mathcal{L}^1+
\mathcal{L}^2)(\mathcal{L}^s+\lambda\mathbf{I})^{-1}]^{-1}$ is a bounded
operator in $C(\T^d)$
for any $\lambda> \lambda_0$. Combining this with the first statement of
 \autoref{l_hold}, we obtain the required statement.
\end{proof}

The operator $\mathcal{L}^\star:= (\mathcal{L}^s+\mathcal{L}^1+\mathcal{L}^2)$
is adjoint to the operator $\mathcal{L}$
defined by
$$
\mathcal{L} q(\zeta)=\int\limits_{\mathbb R^d}\frac{\Lambda(\zeta,\eta)
\big((\Lambda(\eta,\eta))^{-1}q(\eta)-((\Lambda(\zeta,\zeta))^{-1}q(\zeta)\big)}
{|\zeta-\eta|^{d+\alpha}}d\eta.
$$

In the same way as in the proof of \autoref{l_hold} and \autoref{l_ctoc}
one can show that the resolvent $(\mathcal{L}+\lambda)^{-1}$
is a bounded operator from $C(\T^d)$ to $C^\beta(\T^d)$ for
sufficiently large positive $\lambda$.

Considering the properties of the function $\Lambda(\zeta,\zeta)$ and the
definition of operator $\mathcal{L}^\star$ it is
straightforward to see that for sufficiently large $\lambda$ both
$(-L+\lambda)^{-1}$ and  $(-L^\star+\lambda)^{-1}$
are bounded operator from $C(\T^d)$ to $C^\beta(\T^d)$.
Indeed,  taking $\lambda_1>\lambda(\min \Lambda(\xi,\xi))^{-1}$ we have
$$
-L^*+\lambda=\Lambda(\zeta,\zeta)\left(\mathcal{L}^*+\lambda_1+\frac{\lambda}{
\Lambda(\zeta,\zeta)}-\lambda_1\right)
$$
$$
\Lambda(\zeta,\zeta)\left({\bf
I}+\big(\frac{\lambda}{\Lambda(\zeta,\zeta)}-\lambda_1\big)(\mathcal{L}
^*+\lambda_1)^{-1}\right)
(\mathcal{L}^*+\lambda_1)
$$
Since $\|(\mathcal{L}^*+\lambda_1)^{-1}\|_{\mathcal{L}(C(\T^d),C(\mathbb
T^d))}\leq \lambda^{-1}_1$, then
$$
\|\big(\frac{\lambda}{\Lambda(\zeta,\zeta)}-\lambda_1\big)(\mathcal{L}
^*+\lambda_1)^{-1}\|_{\mathcal{L}(C(\T^d),C(\T^d))}\leq
\lambda^{-1}_1\big| \lambda_1-\frac{\lambda}{\max\Lambda(\zeta,\zeta)}\big|<1.
$$
Therefore,
$$
(-L^*+\lambda)^{-1}=(\mathcal{L}^*+\lambda_1)^{-1}
\left({\bf
I}+\big(\frac{\lambda}{\Lambda(\zeta,\zeta)}-\lambda_1\big)(\mathcal{L}
^*+\lambda_1)^{-1}\right)^{-1}
(\Lambda\zeta,\zeta))^{-1}
$$
is a bounded operator from $C(\T^d)$ to $C^\beta(\T^d)$.
The fact that $(-L+\lambda)^{-1}$ is bounded operator from $C(\T^d)$ to
$C^\beta(\T^d)$ can be justified in the same way. This implies in
particular that both  $(-L+\lambda)^{-1}$ and $(-L^*+\lambda)^{-1}$ are compact
operators in
$C(\T^d)$.\\
Also, from the maximum principle it follows that
$\|(-L+\lambda)^{-1}\|_{\mathcal{L}(C(\T^d),C(\T^d))}\leq
\lambda^{-1}$.

By the standard maximum principle arguments, the operator  $(L+\lambda)^{-1}$
maps the set of non-negative continuous non-zero  functions on
$\T^d$ to the set of strictly positive continuous functions on $\mathbb
T^d$. Therefore, the Krein-Rutman theorem applies to the operator
 $(-L+\lambda)^{-1}$.

 It is easy to check that $v=1$ is the principal eigenfunction of
$(-L+\lambda)^{-1}$ and that the corresponding eigenvalue $\mu_0$
 is equal to $\lambda^{-1}$.

By the Krein-Rutman theorem  the adjoint operator
$\big((-L+\lambda)^{-1}\big)^*=(-L^*+\lambda)^{-1}$ maps the cone of
non-negative measures into itself, its principal eigenvalue is $\lambda^{-1}$,
and the corresponding eigenmeasure is positive. Since the adjoint operator maps
the space of continuous functions into itself, it maps the cone of non-negative
continuous functions into itself.

 Applying the maximum principle arguments we conclude that for any non-trivial
continuous non-negative function  $v$ on  $\T^d$
 the function $(-{L}^*+\lambda)^{-1}v$ is strictly positive. Hence the
Krein-Rutman theorem applies to the operator
 $(-{L}^*+\lambda)^{-1}$.  Denote by  $\mu^*_0$  the principal eigenvalue of
this operator and by $p_0(\zeta)$ the corresponding
 continuous eigenfunction. Then the function $p_0(\zeta)$ is positive, and the
measure  $p_0(\zeta)\,d\zeta$ is a positive eigenmeasure. By \cite{KrRu50},
\cite[Chapter 4]{KLS89}  there is only one eigenmeasure in the cone
   of non-negative measures. Therefore, $\mu^*_0=\mu_0=\lambda^{-1}$, and the
function $p_0$ belongs to the kernel
 of the operator ${L}^*$. The proof of \autoref{theo:nonsym-eigenvalue} is
complete.
\end{proof}

\subsection{A priori estimates} \label{subsec:apriori_nonsymm}

Our next goal is to obtain a priori estimates for the solution of equation
\eqref{eq:ori_eqn}.

\begin{proposition}\label{p_nonsym1}
Assume that \eqref{eq:ori_eqn} holds true for some $f\in L^2(\mathbb R^d)$ and
$u^\eps\in \mathcal{D}(L^\eps)$.
Then there exist a constant $c$  such that
$$
\|u^\eps\|_{H^{\alpha/2}(\mathbb R^d)}\leq c(1+\frac1m)\|f\|_{L^{2}(\mathbb
R^d)},
$$
$$
\|u^\eps\|_{L^{2}(\mathbb R^d)}\leq \frac cm\|f\|_{L^{2}(\mathbb R^d)},
$$
The constant $c$ does not depend on $\eps$, neither on $m$.
\end{proposition}
\begin{proof}
Multiplying equation    \eqref{eq:ori_eqn} by $u^\eps(x)p_0(\frac x\eps)$ and
integrating the resulting relation over $\mathbb R^d$ yields
\begin{align}\label{apri1}
\begin{array}{c}
\displaystyle
\int\limits_{\mathbb R^d}\!\!\int\limits_{\mathbb R^d}\frac{\Lambda\big(\frac
x\eps,\frac y\eps\big)p_0\big(\frac x\eps\big)
\big(u^\eps(y)-u^\eps(x)\big)u^\eps(x)}{|x-y|^{d+\alpha}}dydx-m\!\int\limits_{
\mathbb R^d}\!p_0\big(\frac x\eps\big)(u^\eps(x))^2dx
\\[3mm]
\displaystyle
=\int\limits_{\mathbb R^d}p_0\big(\frac x\eps\big)u^\eps(x)f(x)dx
\end{array}
\end{align}
The first term here can be transformed as follows
\begin{align*}
\int\limits_{\mathbb R^d}\!\!\int\limits_{\mathbb R^d} &\frac{\Lambda\big(\frac
x\eps,\frac y\eps\big)p_0\big(\frac x\eps\big)
\big(u^\eps(y)-u^\eps(x)\big)u^\eps(x)}{|x-y|^{d+\alpha}}dydx \\
&=\int\limits_{\mathbb R^d}\!\!\int\limits_{\mathbb R^d}\frac{\Lambda\big(\frac
x\eps,\frac y\eps\big)p_0\big(\frac x\eps\big)
u^\eps(y)u^\eps(x)-\Lambda\big(\frac y\eps,\frac x\eps\big)p_0\big(\frac
y\eps\big)
(u^\eps(x))^2}{|x-y|^{d+\alpha}}dydx
\end{align*}
\begin{align}\label{zeroint}
+\int\limits_{\mathbb R^d}\!\!\int\limits_{\mathbb R^d}\frac{\Lambda\big(\frac
y\eps,\frac x\eps\big)p_0\big(\frac y\eps\big)
(u^\eps(x))^2-\Lambda\big(\frac x\eps,\frac y\eps\big)p_0\big(\frac x\eps\big)
(u^\eps(x))^2}{|x-y|^{d+\alpha}}dydx
\end{align}
$$
=\int\limits_{\mathbb R^d}\!\!\int\limits_{\mathbb R^d}\frac{\Lambda\big(\frac
x\eps,\frac y\eps\big)p_0\big(\frac x\eps\big)
u^\eps(y)u^\eps(x)-\Lambda\big(\frac x\eps,\frac y\eps\big)p_0\big(\frac
x\eps\big)
(u^\eps(y))^2}{|x-y|^{d+\alpha}}dydx
$$
$$
=-\int\limits_{\mathbb R^d}\!\!\int\limits_{\mathbb R^d}\frac{\Lambda\big(\frac
x\eps,\frac y\eps\big)p_0\big(\frac x\eps\big)
\big( u^\eps(y)-u^\eps(x)\big)
u^\eps(y)}{|x-y|^{d+\alpha}}dydx;
$$
here we have used the fact that by \autoref{theo:nonsym-eigenvalue} the
integral in
\eqref{zeroint} is equal to zero.  Considering these equalities one can rewrite
relation \eqref{apri1}
as follows
\begin{align}\label{apri2}
\begin{array}{c}
\displaystyle
\frac12\int\limits_{\mathbb R^d}\!\!\int\limits_{\mathbb
R^d}\frac{\Lambda\big(\frac x\eps,\frac y\eps\big)p_0\big(\frac x\eps\big)
\big(u^\eps(y)-u^\eps(x)\big)^2}{|x-y|^{d+\alpha}}dydx+m\!\int\limits_{\mathbb
R^d}\!p_0\big(\frac x\eps\big)(u^\eps(x))^2dx
\\[3mm]
\displaystyle 
=-\int\limits_{\mathbb R^d}p_0\big(\frac x\eps\big)u^\eps(x)f(x)dx
\end{array}
\end{align}
By \autoref{theo:nonsym-eigenvalue} the function $p_0$ satisfies the estimates
$0<p_-\leq
p_0(z)\leq p^+$. Therefore, we have
$$
\|u^\eps\|_{H^{\alpha/2}(\mathbb R^d)}\leq c(1+\frac1m)\|f\|_{L^{2}(\mathbb
R^d)},\qquad
\|u^\eps\|_{L^{2}(\mathbb R^d)}\leq \frac{p^+}{m p_-}\|f\|_{L^{2}(\mathbb R^d)}.
$$
This completes the proof of Proposition.
\end{proof}

From the estimates of \autoref{p_nonsym1} one can easily deduce that the
set of positive real numbers belongs to the
resolvent set of operator ${L}^\eps$.  In particular, equation
\eqref{eq:ori_eqn}
is well posed and it has a unique solution
$u^\eps\in \mathcal{D}(L^\eps)$.

\subsection{Passage to the limit} \label{subsec:limit_nonsymm}

According to the estimates of \autoref{p_nonsym1} the family $u^\eps$
converges for a subsequence, as $\eps\to0$,
to a function $u^0\in H^{\alpha/2}(\mathbb R^d)$, weakly in
$H^{\alpha/2}(\mathbb R^d)$. Furthermore,
$u^\eps\to u^0$ strongly in $L^2$ on any compact set in $\mathbb R^d$.

In order to characterize the function $u^0$ we multiply equation
\eqref{eq:ori_eqn}
by a test function $p_0(\frac x\eps)\varphi(x)$ with
$\varphi\in\C_0^\infty(\mathbb R^d)$
and integrate the resulting relation in $\mathbb R^d$. We have
\begin{align}\label{limpas1}
\begin{array}{c}
\displaystyle
\!\int\limits_{\mathbb R^d}\!\!\int\limits_{\mathbb
R^d}\!\frac{\Lambda\big(\frac x\eps,\frac y\eps\big)p_0\big(\frac x\eps\big)
\big(u^\eps(y)-u^\eps(x)\big)\varphi(x)}{|x-y|^{d+\alpha}}dydx-m\!\int\limits_{
\mathbb R^d}\!\!p_0\big(\frac x\eps\big)u^\eps(x)\varphi(x)dx
\\[3mm]
\displaystyle
=\int\limits_{\mathbb R^d}p_0\big(\frac x\eps\big)\varphi(x)f(x)dx
\end{array}
\end{align}
In the same way as in the proof of  \autoref{p_nonsym1} one can show
that
$$
\int\limits_{\mathbb R^d}\!\!\int\limits_{\mathbb R^d}\!\frac{\Lambda\big(\frac
x\eps,\frac y\eps\big)p_0\big(\frac x\eps\big)
\big(u^\eps(y)-u^\eps(x)\big)\varphi(x)}{|x-y|^{d+\alpha}}dydx
$$
$$
=-\int\limits_{\mathbb R^d}\!\!\int\limits_{\mathbb
R^d}\!\frac{\Lambda\big(\frac x\eps,\frac y\eps\big)p_0\big(\frac x\eps\big)
\big(\varphi(y)-\varphi(x)\big)u^\eps(y)}{|x-y|^{d+\alpha}}dydx.
$$
We represent $\mathbb R^d\times\mathbb R^d$  as the union of two sets
\begin{align}\label{divi_spa_nsym}
\mathbb R^d\times\mathbb R^d= G^\delta_4\cup G^\delta_5
\end{align}
with
\begin{align}\label{defn_g45}
G^\delta_4=\{(x,y)\,:\,|x-y|\leq\delta\},\quad
G^\delta_5=\{(x,y)\,:\,|x-y|>\delta\}.
\end{align}
Denote
$$
\mathcal{K}^\eps_\delta(x,y)=\frac{\Lambda\big(\frac x\eps,\frac
y\eps\big)p_0\big(\frac x\eps\big)
\big(\varphi(y)-\varphi(x)\big)}{|x-y|^{d+\alpha}}\ {\bf 1}_\delta(x-y),
$$
where ${\bf 1}_\delta(z)$ is an indicator function of the ball $\{z\in\mathbb
R^d\,:\,|z|\leq\delta\}$.
It is easy to check that
$$
0\leq \mathcal{K}^\eps_\delta(x,y)\leq C_\varphi |x-y|^{1-d-\alpha}{\bf
1}_\delta(x-y).
$$
Since the integral
$$
\int\limits_{\mathbb R^d}|z|^{1-d-\alpha}{\bf 1}_\delta(z)dz
$$
tends to zero, as $\delta\to 0$,
we have
\begin{align}\label{part1}
\int\limits_{\mathbb R^d}dx \bigg(\int\limits_{\mathbb
R^d}\mathcal{K}^\eps_\delta(x,y)u^\eps(y)dy\bigg)^2\leq
C(\delta)\|u^\eps\|^2_{L^ 2(\mathbb R^d)},
\end{align}
where $C(\delta)\to0$, as $\delta\to 0$.  On the set $G^\delta_5$ the kernel
is bounded. Therefore,
$$
\int\limits_{G^\delta_5}\frac{\Lambda\big(\frac x\eps,\frac
y\eps\big)p_0\big(\frac x\eps\big)
\big(\varphi(y)-\varphi(x)\big)u^\eps(y)}{|x-y|^{d+\alpha}}dydx\ \to \
\int\limits_{G^\delta_5}\frac{\langle\Lambda p_0\rangle
\big(\varphi(y)-\varphi(x)\big)u^0(y)}{|x-y|^{d+\alpha}}dydx
$$
where
$$
\langle\Lambda p_0\rangle=\int\limits_{\T^d\times \T^d}
\Lambda(\zeta,\eta)p_0(\zeta)d\zeta d\eta.
$$
Combining this convergence with \eqref{part1} we conclude that

\begin{align*}
\int\limits_{\mathbb R^d}&\!\!\int\limits_{\mathbb R^d}\!\frac{\Lambda\big(\frac
x\eps,\frac y\eps\big)p_0\big(\frac x\eps\big)
\big(\varphi(y)-\varphi(x)\big)u^\eps(y)}{|x-y|^{d+\alpha}}dydx \\
&\longrightarrow
\int\limits_{\mathbb R^d}\!\!\int\limits_{\mathbb R^d}\!\frac{\langle\Lambda
p_0\rangle
\big(\varphi(y)-\varphi(x)\big)u^0(y)}{|x-y|^{d+\alpha}}dydx,
\end{align*}
as $\eps\to0$. Therefore,

\begin{align*}
\int\limits_{\mathbb R^d}&\!\!\int\limits_{\mathbb R^d}\!\frac{\Lambda\big(\frac
x\eps,\frac y\eps\big)p_0\big(\frac x\eps\big)
\big(u^\eps(y)-u^\eps(x)\big)\varphi(x)}{|x-y|^{d+\alpha}}dydx \\
&\longrightarrow
\int\limits_{\mathbb R^d}\!\!\int\limits_{\mathbb R^d}\!\frac{\langle\Lambda
p_0\rangle
\big(u^0(y)-u^0(x)\big)\varphi(x)}{|x-y|^{d+\alpha}}dydx \,.
\end{align*}
Passing to the limit in \eqref{limpas1} yields

\begin{align*}
\int\limits_{\mathbb R^d} &\!\!\int\limits_{\mathbb R^d}\!\frac{\langle\Lambda
p_0\rangle
\big(u^0(y)-u^0(x)\big)\varphi(x)}{|x-y|^{d+\alpha}}dydx -m \int\limits_{\mathbb
R^d}\langle p_0\rangle u^0(x)\varphi(x)dx \\
&= \int\limits_{\mathbb R^d}\langle p_0\rangle f(x))\varphi(x)dx.
\end{align*}

It remains to divide this equation by $\langle p_0\rangle $ and denote
$\Lambda^{\rm eff}=\langle p_0\rangle^{-1}\langle\Lambda p_0\rangle$. Then
the limit equation takes the form

\begin{align*}
\int\limits_{\mathbb R^d}\!\!\int\limits_{\mathbb R^d}\!\frac{\Lambda^{\rm eff}
\big(u^0(y)-u^0(x)\big)\varphi(x)}{|x-y|^{d+\alpha}}dydx -m\int\limits_{\mathbb
R^d} u^0(x)\varphi(x)dx=
\int\limits_{\mathbb R^d} f(x)\varphi(x)dx.
\end{align*}

Finally, we can complete the proof of \autoref{theo:nonsym}.   The weak
convergence in $H^{\alpha/2}(\mathbb R^ d)$ has already been proved. The
convergence in   $L^2(\mathbb R^ d)$ can be shown in exactly the same way as in
the proof of \autoref{theo:symm-kernels}.



{\small
\bibliographystyle{alpha}
\bibliography{kapizh}

\begin{thebibliography}{FBRS17}

\bibitem[Ari09]{MR2560294}
Mariko Arisawa.
\newblock Homogenization of a class of integro-differential equations with
  {L}\'evy operators.
\newblock {\em Comm. Partial Differential Equations}, 34(7-9):617--624, 2009.

\bibitem[Ari12]{MR2981018}
Mariko Arisawa.
\newblock Homogenizations of integro-differential equations with {L}\'evy
  operators with asymmetric and degenerate densities.
\newblock {\em Proc. Roy. Soc. Edinburgh Sect. A}, 142(5):917--943, 2012.

\bibitem[Bas09]{MR2555009}
Richard~F. Bass.
\newblock Regularity results for stable-like operators.
\newblock {\em J. Funct. Anal.}, 257(8):2693--2722, 2009.

\bibitem[BCCI14]{MR3194684}
Guy Barles, Emmanuel Chasseigne, Adina Ciomaga, and Cyril Imbert.
\newblock Large time behavior of periodic viscosity solutions for uniformly
  parabolic integro-differential equations.
\newblock {\em Calc. Var. Partial Differential Equations}, 50(1-2):283--304,
  2014.

\bibitem[BGG18]{BGG18}
L.~B\u{a}lilescu, A.~Ghosh, and T~Ghosh.
\newblock Homogenization for non-local elliptic operators in both perforated
  and non-perforated domains, 2018.
\newblock https://arxiv.org/pdf/1805.06264.

\bibitem[BJP99]{MR1798387}
V.~Berdichevsky, V.~Jikov, and G.~Papanicolaou, editors.
\newblock {\em Homogenization}, volume~50 of {\em Series on Advances in
  Mathematics for Applied Sciences}.
\newblock World Scientific Publishing Co., Inc., River Edge, NJ, 1999.
\newblock In memory of Serguei Kozlov.

\bibitem[Bra05]{braides2005gamma}
A.~Braides.
\newblock {\em Gamma -convergence for beginners}.
\newblock Oxford Lecture Series in Mathematics and its Applications. Oxford
  University Press, 2005.

\bibitem[CD99]{MR1765047}
Doina Cioranescu and Patrizia Donato.
\newblock {\em An introduction to homogenization}, volume~17 of {\em Oxford
  Lecture Series in Mathematics and its Applications}.
\newblock The Clarendon Press, Oxford University Press, New York, 1999.

\bibitem[CPS07]{MR2337848}
G.~A. Chechkin, A.~L. Piatnitski, and A.~S. Shamaev.
\newblock {\em Homogenization}, volume 234 of {\em Translations of Mathematical
  Monographs}.
\newblock American Mathematical Society, Providence, RI, 2007.
\newblock Methods and applications, Translated from the 2007 Russian original
  by Tamara Rozhkovskaya.

\bibitem[FBRS17]{MR3668594}
Juli\'an Fern\'andez~Bonder, Antonella Ritorto, and Ariel~Martin Salort.
\newblock {$H$}-convergence result for nonlocal elliptic-type problems via
  {T}artar's method.
\newblock {\em SIAM J. Math. Anal.}, 49(4):2387--2408, 2017.

\bibitem[Foc10]{MR2729014}
Matteo Focardi.
\newblock Aperiodic fractional obstacle problems.
\newblock {\em Adv. Math.}, 225(6):3502--3544, 2010.

\bibitem[Fra06]{MR2232732}
Brice Franke.
\newblock The scaling limit behaviour of periodic stable-like processes.
\newblock {\em Bernoulli}, 12(3):551--570, 2006.

\bibitem[Fra07a]{MR2331266}
Brice Franke.
\newblock Correction to: ``{T}he scaling limit behaviour of periodic
  stable-like processes'' [{B}ernoulli {\bf 12} (2006), no. 3, 551--570;
  mr2232732].
\newblock {\em Bernoulli}, 13(2):600, 2007.

\bibitem[Fra07b]{MR2362589}
Brice Franke.
\newblock A functional non-central limit theorem for jump-diffusions with
  periodic coefficients driven by stable {L}\'evy-noise.
\newblock {\em J. Theoret. Probab.}, 20(4):1087--1100, 2007.

\bibitem[FT94]{MR1784743}
Tsukasa Fujiwara and Matsuyo Tomisaki.
\newblock Martingale approach to limit theorems for jump processes.
\newblock {\em Stochastics Stochastics Rep.}, 50(1-2):35--64, 1994.

\bibitem[HIT77]{MR0488335}
Masayuki Horie, Tadataka Inuzuka, and Hiroshi Tanaka.
\newblock Homogenization of certain one-dimensional discontinuous {M}arkov
  processes.
\newblock {\em Hiroshima Math. J.}, 7(2):629--641, 1977.

\bibitem[JKOn94]{JKO94}
V.~V. Jikov, S.~M. Kozlov, and O.~A. Ole\u\i~nik.
\newblock {\em Homogenization of differential operators and integral
  functionals}.
\newblock Springer-Verlag, Berlin, 1994.
\newblock Translated from the Russian by G. A. Yosifian [G. A. Iosifyan].

\bibitem[KLS89]{KLS89}
M.~A. Krasnosel'skij, Je.~A. Lifshits, and A.~V. Sobolev.
\newblock {\em Positive linear systems}, volume~5 of {\em Sigma Series in
  Applied Mathematics}.
\newblock Heldermann Verlag, Berlin, 1989.
\newblock The method of positive operators, Translated from the Russian by
  J\"urgen Appell.

\bibitem[KR50]{KrRu50}
M.~G. Kre\u{\i}n and M.~A. Rutman.
\newblock Linear operators leaving invariant a cone in a {B}anach space.
\newblock {\em Amer. Math. Soc. Translation}, 1950(26):128, 1950.

\bibitem[PZ17]{PZ2017}
A.~Piatnitski and E.~Zhizhina.
\newblock Periodic homogenization of nonlocal operators with a convolution-type
  kernel.
\newblock {\em SIAM J. Math. Anal.}, 49(1):64--81, 2017.

\bibitem[RS75]{ReSi75}
Michael Reed and Barry Simon.
\newblock {\em Methods of modern mathematical physics. {II}. {F}ourier
  analysis, self-adjointness}.
\newblock Academic Press [Harcourt Brace Jovanovich, Publishers], New
  York-London, 1975.

\bibitem[RV09]{MR2565557}
R\'emi Rhodes and Vincent Vargas.
\newblock Scaling limits for symmetric {I}t\^o-{L}\'evy processes in random
  medium.
\newblock {\em Stochastic Process. Appl.}, 119(12):4004--4033, 2009.

\bibitem[San16]{San16}
Nikola Sandri\'c.
\newblock Homogenization of periodic diffusion with small jumps.
\newblock {\em J. Math. Anal. Appl.}, 435(1):551--577, 2016.

\bibitem[Sat13]{Sat13}
Ken-iti Sato.
\newblock {\em L\'evy processes and infinitely divisible distributions},
  volume~68 of {\em Cambridge Studies in Advanced Mathematics}.
\newblock Cambridge University Press, Cambridge, 2013.
\newblock Translated from the 1990 Japanese original, Revised edition of the
  1999 English translation.

\bibitem[Sch10]{MR2733264}
Russell~W. Schwab.
\newblock Periodic homogenization for nonlinear integro-differential equations.
\newblock {\em SIAM J. Math. Anal.}, 42(6):2652--2680, 2010.

\bibitem[Sch13]{MR3009077}
Russell~W. Schwab.
\newblock Stochastic homogenization for some nonlinear integro-differential
  equations.
\newblock {\em Comm. Partial Differential Equations}, 38(2):171--198, 2013.

\bibitem[Tar09]{MR2582099}
Luc Tartar.
\newblock {\em The general theory of homogenization}, volume~7 of {\em Lecture
  Notes of the Unione Matematica Italiana}.
\newblock Springer-Verlag, Berlin; UMI, Bologna, 2009.
\newblock A personalized introduction.

\bibitem[Tom92]{MR1154844}
Matsuyo Tomisaki.
\newblock Homogenization of c\`adl\`ag processes.
\newblock {\em J. Math. Soc. Japan}, 44(2):281--305, 1992.

\bibitem[Zhi03]{Zhikov2004}
V.~V. Zhikov.
\newblock On two-scale convergence.
\newblock {\em Tr. Semin. im. I. G. Petrovskogo}, (23):149--187, 410, 2003.
\newblock translation in J. Math. Sci. (N. Y.) 120 (2004), no. 3, 1328--1352.

\end{thebibliography}
}

\end{document}